\documentclass[12pt]{amsart}
\usepackage[utf8]{inputenc}

\usepackage{amsmath,amsthm, amssymb}
\usepackage{hyperref}
\usepackage{relsize}
\usepackage[capitalise, noabbrev]{cleveref}
\usepackage[hideerrors]{xcolor}
\usepackage{pgf,tikz}
\usepackage{bold-extra}
\usepackage{multirow}
\usepackage{tabularx}

\usepackage{bm}

\newcommand{\col}{\textsc{Col }}
\newcommand{\Col}{\textsc{Col}}
\newcommand{\snort}{\textsc{Snort }}
\newcommand{\Snort}{\textsc{Snort}}
\newcommand{\cis}{\textsc{Cis }}
\newcommand{\Cis}{\textsc{Cis}}

\newcommand{\Encol}{\textsc{EnCol}}

\newcommand{\Ensnort}{\textsc{EnSnort}}

\newcommand{\Encis}{\textsc{EnCis}}

\theoremstyle{definition} 
	\newtheorem{definition}{Definition}[section]
	\newtheorem{example}[definition]{Example}

\theoremstyle{plain}
	\newtheorem{theorem}[definition]{Theorem}
	
	\newtheorem{proposition}[definition]{Proposition}
	
	\newtheorem{conjecture}[definition]{Conjecture}

\usepackage{xspace}
\makeatletter
\DeclareRobustCommand\onedot{\futurelet\@let@token\@onedot}
\def\@onedot{\ifx\@let@token.\else.\null\fi\xspace}
 
\def\ie{{i.e}\onedot}

\def\etal{{et al}\onedot}
\makeatother

\usepackage[
left=1.5in,
right=1.5in,
top=1.5in,
bottom=1.5in]{geometry}

\begin{document}

\author{Svenja Huntemann}
\address{Department of Mathematical and Physical Sciences\\ Concordia University of Edmonton\\ Edmonton, AB, Canada}
\author{Lexi Nash}

\thanks{The second author's research was supported by the Natural Sciences and Engineering Research Council of Canada (funding reference number USRA – 562933 - 2021).}

\title{The Polynomial Profile of Distance Games on Paths and Cycles}
\begin{abstract}
Distance games are games played on graphs in which the players alternately colour vertices, and which vertices can be coloured only depends on the distance to previously coloured vertices. The polynomial profile encodes the number of positions with a fixed number of vertices from each player. We extend previous work on finding the polynomial profile of several distance games (\Col, \Snort, and \Cis) played on paths. We give recursions and generating functions for the polynomial profiles of generalizations of these three games when played on paths. We also find the polynomial profile of \cis played on cycles and the total number of positions of \col and \snort on cycles, as well as pose a conjecture about the number of positions when playing \col and \snort on complete bipartite graphs.
\end{abstract}


\maketitle

\section{Introduction}

A combinatorial game is a  game with two players, Left and Right, with perfect information, no elements of chance, and the players must alternate turns, such as \textsc{Chess} or \textsc{Checkers}. In this paper, we will be enumerating the positions in several distance games. These are games in which the two players place pieces on empty vertices of the board, a finite graph, with the placement of a piece only being restricted by the distance to previously played pieces. Two well-known examples of distance games are \snort and \Col, where pieces cannot be placed adjacent to an opponent's piece, respectively one of their own pieces.

A relatively new topic of interest in combinatorial game theory is the enumeration of positions. Early work has focused on specific types of positions, such as \textsc{Go} end positions \cite{farr_2003,farr_schmidt_2008,TrompFarneback2007} and second-player win position for some lesser known games \cite{hetyei_2009,timber_2013}. Recently, \textsc{Domineering} positions, as well as specific types of positions, were counted in \cite{counting}. Although taking a graph theory view, positions of several other games played on grids are enumerated in \cite{independent1,independent2}. Our work is motivated by Brown \etal \cite{profiles}, who introduced the polynomial profile of a game, which is the generating polynomial for the number of positions given the number of Left and Right pieces placed. They found generating functions for the polynomial profiles of several games played on paths, including \Snort, \Col, and the game \Cis. We extend their work and will find the polynomial profiles of generalizations of these three games played on paths (\cref{Paths}), and consider the original three games when played on cycles and complete bipartite graphs (\cref{csc}).

In the next section, we provide some background and previous results from this area of combinatorial game theory. This includes a few definitions and descriptions of different placement games. We will conclude this paper with several questions for future work in \cref{ques}.

\section{Background}\label{background}


Many combinatorial games are ones in which the players place pieces on a board without moving or removing them later. These are known as placement or pen-and-paper games. We will be studying several games that belong to a subclass of placement games. Note that for placement games, placing a piece is equivalent to colouring the corresponding vertices of the board. For the placement, Left will colour a vertex bLue and Right will colour in Red.

\begin{definition}[\cite{distance}]
\emph{Distance games} are a subclass of combinatorial games played on finite graphs, with each game uniquely identified by a pair of sets $(S,D)$. On their turn, a player will colour in an empty vertex that is not distance $s\in S$ from a piece of the same colour or distance $d\in D$ from a piece of the opposite (different) colour. 
\end{definition}

Although the rules of a distance game are deceptively easy to describe, several games have been extensively studied and are still unsolved. For example, their computational complexity has been studied in \cite{schaefer_78,KeepingDistance,HugganStevens2016}. 

\begin{definition}
The following are some previously studied distance games. 
\begin{itemize}
	\item \emph{\col}\cite{winning_ways} is the distance game where $S=\{1\}$ and $D=\emptyset$. In other words, pieces of different colours can be adjacent, but pieces of the same colour cannot. 
	\item \emph{\snort}\cite{winning_ways} is the distance game where $S=\emptyset$ and $D=\{1\}$. In other words, pieces of the same colour can be adjacent, but pieces of different colours cannot.
	\item \emph{\cis} \cite{profiles} is the distance game where $S=D=\{1\}$. In other words, no two pieces can be adjacent, no matter the colour.
\end{itemize}
\end{definition}

Note that playing \cis is equivalent to playing \textsc{NodeKayles}. If a player can place a piece on a vertex in either game, then so can the other player. The difference between these games is that in \textsc{NodeKayles} both players colour using the same colour, while in \cis they use different colours. This is not relevant to the game play, but will be important for enumeration.

We will consider generalizations of the above games, the first two of which were defined in \cite{KeepingDistance}.
\begin{definition} 
\begin{itemize}
	\item \Encol$(k)$ is the distance game where $S=\{1,\ldots,k\}$ and $D=\emptyset$.
	\item \Ensnort$(k)$ is the distance game where $S=\emptyset$ and $D=\{1,\ldots,k\}$. 
	\item \Encis$(k)$ is the distance game where $S=D=\{1,\ldots,k\}$.
	\item $\Cis_2$ is the distance game where $S=D=\{2\}$.
\end{itemize}
\end{definition}

We will count the positions in \Snort, \Col, and \cis played on cycles, stars, and complete bipartite graphs, as well as their generalization when played on paths. To enumerate the positions of these distance games we use a generating function called the polynomial profile. 
\begin{definition}[\cite{profiles}]
The \emph{polynomial profile} of the game $G$ played on the board $B$ with $n$ vertices is the bivariate polynomial \[P_{G,B}(x,y)=\sum_{k=0}^n \sum_{j=0}^k f_{j,k-j}x^jy^{k-j}\] where $f_{j,j-k}$ is the number of position with $j$ Left pieces and $k-j$ Right pieces. 
\end{definition}
Setting $x=y$ we get the univariate polynomial  \[P_{G,B}(x):=P_{G,B}(x,x)=\sum_{i=0}^nc_ix^i,\] where $c_i$ is the number of positions with exactly $i$ pieces. Finally, the total number of legal positions can be obtained by setting $x=y=1$. 

The polynomial profile counts the number of positions without assuming alternating play. In many combinatorial games, including distance games, the board naturally breaks into smaller, independent components as game play progresses. On their turn, a player then chooses which component to play in and makes their move there. In combinatorial game theory, this is called the \emph{disjunctive sum} of the components. Although game play in the entire game is alternating, in a component it can be non-alternating. In many cases in combinatorial game theory, including enumeration of positions, it helps to assume that a game is a component of a larger game and the condition of alternating play is dropped. If desired, we can find the number of positions restricted to alternating play from the polynomial profile by taking only the terms where the exponents on $x$ and $y$ differ by at most 1. More information on combinatorial game theory and common techniques can be found in \cite{winning_ways,Siegel_2013}.

To illustrate these concept, we will look at a relatively simple example.
\begin{example}
Consider \cis played on $P_4$. There is one empty position, four positions with a single Left or Right piece each, three positions with two pieces of the same colour, and six positions with one piece each by Left and Right (shown in \cref{6xy}). Thus the polynomial profile is 
\[P_{\textsc{Cis},P_4}(x,y)=1+4x+4y+3x^2+6xy+3y^2.\] The univariate polynomial is \[P_{\textsc{Cis},P_4}(x)=1+8x+12x^2\] and the total number of positions is $P_{\textsc{Cis},P_4}(1)=21$.

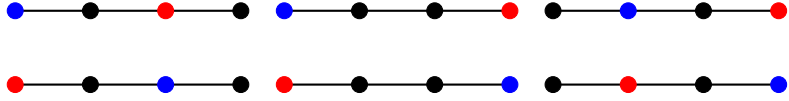
\begin{figure}[!ht]
	\begin{center}
		\begin{tabular}{c c c}
    			\begin{tikzpicture}
    				\draw[black, thick] (0,0)--(3,0);
				\filldraw [ blue] (0,0) circle (3pt);
				\filldraw [ black] (1,0) circle (3pt);
				\filldraw [ red] (2,0) circle (3pt);
				\filldraw [ black] (3,0) circle (3pt);
				\end{tikzpicture} &  \begin{tikzpicture}
				\draw[black, thick] (0,0)--(3,0);
				\filldraw [ blue] (0,0) circle (3pt);
				\filldraw [ black] (1,0) circle (3pt);
				\filldraw [ black] (2,0) circle (3pt);
				\filldraw [ red] (3,0) circle (3pt);
				\end{tikzpicture} & \begin{tikzpicture}
				\draw[black, thick] (0,0)--(3,0);
				\filldraw [ black] (0,0) circle (3pt);
				\filldraw [ blue] (1,0) circle (3pt);
				\filldraw [ black] (2,0) circle (3pt);
				\filldraw [ red] (3,0) circle (3pt);
			\end{tikzpicture} \\
			 \\
    			\begin{tikzpicture}
    				\draw[black, thick] (0,0)--(3,0);
				\filldraw [ red] (0,0) circle (3pt);
				\filldraw [ black] (1,0) circle (3pt);
				\filldraw [ blue] (2,0) circle (3pt);
				\filldraw [ black] (3,0) circle (3pt);
				\end{tikzpicture} &  \begin{tikzpicture}
				\draw[black, thick] (0,0)--(3,0);
				\filldraw [ red] (0,0) circle (3pt);
				\filldraw [ black] (1,0) circle (3pt);
				\filldraw [ black] (2,0) circle (3pt);
				\filldraw [ blue] (3,0) circle (3pt);
				\end{tikzpicture} & \begin{tikzpicture}
				\draw[black, thick] (0,0)--(3,0);
				\filldraw [ black] (0,0) circle (3pt);
				\filldraw [ red] (1,0) circle (3pt);
				\filldraw [ black] (2,0) circle (3pt);
				\filldraw [ blue] (3,0) circle (3pt);
			\end{tikzpicture}
		\end{tabular}
		\caption{Possible positions of \cis on $P_4$ with one Left piece and one Right piece played.}
		\label{6xy}
	\end{center}
\end{figure}

Taking only the relevant terms for alternating play from $P_{\textsc{Cis},P_4}(x,y)$, we get that the positions in alternating play are enumerated by $1+4x+4y+6xy$.
\end{example}

We can find the polynomial profile of a game by using the auxiliary board.
\begin{definition}
The \emph{auxiliary board} $\Gamma_{G,B}$ of a distance game $G$ on a board $B$ is the graph that represents all minimal illegal moves. The vertex set of $\Gamma_{G,B}$ is given by $V(\Gamma)=V(B) \times \{1,2\}$ where the vertex $(x_i,1)$ represents Left moving in vertex $x_i$ in $B$ and similarly $(x_j,2)$ is a move by Right in vertex $x_j$. Two vertices $(x_i,a)$ and $(x_j,b)$ are adjacent if the corresponding moves are at an illegal distance.
\end{definition}
Note that the auxiliary board can be generalized to many other placement games, and the resulting simplicial complex is known as the illegal complex \cite{simplicialComplexes}.

\begin{example}
Consider \col played on $C_4$. The auxiliary board $\Gamma_{\Col,C_4}$, shown in \cref{colC4}, has vertex set $V(G_{\Col,C_4})=\{x_1,x_2,x_3,x_4\}\times \{1,2\}$. Vertices $(x_i,p)$ and $(x_j,q)$ are adjacent if $x_i \sim x_j$ and $p=q$ or if $i=j$ and $p\neq q$. 

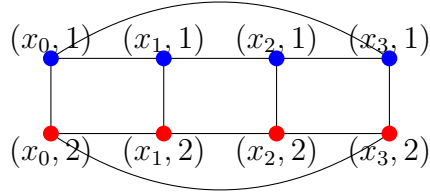
\begin{figure}[!ht]
	\begin{center}
		\begin{tikzpicture}
	    \draw (0,0) .. controls (1.5,-1) and (3,-1) .. (4.5,0);
	    \draw (0,1) .. controls (1.5,2) and (3,2) .. (4.5,1);
	    \draw(0,0)--(4.5,0);
	    \draw(0,1)--(4.5,1);
	    \draw(0,0)--(0,1);
	    \draw(1.5,0)--(1.5,1);
	    \draw(3,0)--(3,1);
	    \draw(4.5,0)--(4.5,1);
	    \foreach \x in {0,1.5,3,4.5}
		{
	    \fill [ red] (\x,0) circle (3pt);
	    \fill [ blue] (\x,1) circle (3pt);
	    }
	    \node at (0,1.25) {{$(x_0,1)$}};
	    \node at (1.5,1.25) {{$(x_1,1)$}};
	    \node at (3,1.25) {{$(x_2,1)$}};
	    \node at (4.5,1.25) {{$(x_3,1)$}};
	    \node at (0,-0.25) {{$(x_0,2)$}};
	    \node at (1.5,-0.25) {{$(x_1,2)$}};
	    \node at (3,-0.25) {{$(x_2,2)$}};
	    \node at (4.5,-0.25) {{$(x_3,2)$}};
		\end{tikzpicture}
	\end{center}
	\caption{Auxiliary board for \col on $C_4$.}\label{colC4}
\end{figure}

From \cref{colC4} we can see that there are eight independent sets of size 1, four of which are for Left and four for Right, this gives us the terms $4x$ and $4y$ in the polynomial profile. We can also see that there are twelve independent sets with one blue and one red piece, which gives us the term $12xy$. We do this for every possible combinations of blue and red pieces to get the full polynomial \[P_{\Col,C_4}(x,y)=1+4x+4y+2x^2+12xy+2y^2+4x^2y+4xy^2+2x^2y^2.\] 
\end{example}
\begin{example}
Consider \snort played on $C_4$ . The auxiliary board $\Gamma_{\Snort,C_4}$, shown in \cref{snortC4}, has vertex set $V(G_{\Snort,C_4})=\{x_1,x_2,x_3,x_4\}\times \{1,2\}$. Vertices $(x_i,p)$ and $(x_j,q)$ are adjacent if $i=j$ and $p\not=q$ or if both $x_i\sim x_j$ and $p\not=q$.
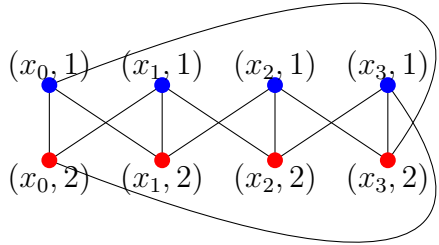
\begin{figure}[!ht]
	\begin{center}
		\begin{tikzpicture}
		
	    \draw(0,0)--(0,1);
	    \draw(1.5,0)--(1.5,1);
	    \draw(3,0)--(3,1);
	    \draw(4.5,0)--(4.5,1);
	    \draw(0,0)--(1.5,1)--(3,0)--(4.5,1);
	    \draw(0,1)--(1.5,0)--(3,1)--(4.5,0);
	    \draw(0,0) .. controls (5,-2) and (6,-1) .. (4.5,1);
	    \draw(0,1) .. controls (5,3) and (6,2) .. (4.5,0);
	    \foreach \x in {0,1.5,3,4.5}
		{
	    \fill [ red] (\x,0) circle (3pt);
	    \fill [ blue] (\x,1) circle (3pt);
	    }
	    \node at (0,1.25) {$(x_0,1)$};
	    \node at (1.5,1.25) {{$(x_1,1)$}};
	    \node at (3,1.25) {{$(x_2,1)$}};
	    \node at (4.5,1.25) {{$(x_3,1)$}};
	    \node at (0,-0.25) {{$(x_0,2)$}};
	    \node at (1.5,-0.25) {{$(x_1,2)$}};
	    \node at (3,-0.25) {{$(x_2,2)$}};
	    \node at (4.5,-0.25) {{$(x_3,2)$}};
		\end{tikzpicture}
	\end{center}
	\caption{Auxiliary board for \snort on $C_4$.}	\label{snortC4}
\end{figure}

From \cref{snortC4} we can see that there are two independent sets of size 4, one for Left and one for Right, this gives us the terms $x^4$ and $y^4$ in the polynomial profile. We can also see that there are eight independent sets of size 3, four of which are for Left and four for Right, this gives us the entries $4x^3$ and $4y^3$. We do this for every possible combinations of blue and red pieces to get the full polynomial \[P_{\Snort, C_4}=1+4x+4y+6x^2+4xy+6y^2+4x^3+4y^3+x^4+y^4.\]

\end{example}

We will now turn our attention to determining the polynomial profiles in greater generality.

\section{Distance Games on Paths}\label{Paths}
The positions of \Col, \Snort, and \cis played on paths were enumerated in \cite{profiles}. The generating functions for these three games were found to be 
\begin{align*}
GF_{\Col}(x,y,t)&=\frac{(1+xt)(1+yt)}{1-(xyt^2+t(1+xt)(1+yt))}\\
GF_{\Snort}(x,y,t)&=\frac{1-xyt^2}{1-(xt+yt+xyt^2+t(1-xyt^2))}\\
GF_{\Cis}(x,y,t)&=\frac{1+xt+yt}{1-t-xt^2-yt^2},
\end{align*}
where the coefficient of $t^n$ is the polynomial profile when playing on a path of $n$ vertices. Note that there was a typo in the generating function of \cis in \cite{profiles}. We have given the corrected function here.

We will now generalize this and enumerate the positions of various other distance games when played on paths.

As a step to find the generating function we take into account the empty vertices in positions by using an $e$ to represent them, we then set $e=1$ to get the polynomial profile. 
\subsection{\Encis($k$)}

In this section we find the recursions for the polynomial profile and the number of positions of \Encis($k$) played on paths, as well as the generating function.
\begin{proposition}
The polynomial profile for the distance game $\Encis(k)$ played on $P_n$ for $n\leq k$ is
\[P_{\Encis(k),P_n}(x,y)=1+nx+ny\]
and for $n>k$ it is recursively given by 
\[P_{\Encis(k),P_n}(x,y)= P_{\Encis(k),P_{n-1}}(x,y)+(x+y)P_{\Encis(k),P_{n-(k+1)}}(x,y).
\]

The total number of positions is recursively given by \[P_{\Encis(k),P_n}(1)=P_{\Encis(k),P_{n-1}}(1)+2P_{\Encis(k),P_{n-(k+1)}}(1)\]
with $P_{\Encis(k),P_n}(1)=2n+1$ for $n\leq k$.

%
\end{proposition}
\begin{proof}
When $n\leq k$, no pair of pieces can be placed. Therefore the only positions are the empty board and the $n$ positions each with a single blue or red piece, giving the initial terms.

For the recursion, consider the leftmost vertex of the path.

If this vertex is uncoloured, the rest of the $n-1$ vertices can be any legal position on $P_{n-1}$.

If this vertex is coloured blue or red, then the next $k$ vertices must be empty, which leaves any legal position on $P_{n-(k+1)}$ vertices.

This gives the desired recursion. Setting $x=y=1$ gives the recursion for the total number of positions.
\end{proof}

For large $k$, this recursion will require a large amount of initial terms. Although we are not able to give a closed form of the polynomial profile, we will now turn to determining the generating function, which is
\[GF_{G, P_n}(x,y,t)=\sum_{n=0}^\infty P_{G,P_n}(x,y)t^n.\]
As an in-between step, we will often consider a refined polynomial profile $P_{G,P_n}(e,x,y)$ where the exponent on $e$ indicates the number of empty vertices, \ie the degree of every term will be $n$. The generating function for this polynomial is denoted $GF_{G, P_n}(e,x,y,t)$ and we have $P_{G,P_n}(x,y)=P_{G,P_n}(1,x,y)$ and $GF_{G, P_n}(x,y,t)=GF_{G, P_n}(1,x,y,t)$.

\begin{proposition}
The generating function for the polynomial profile of $\Encis(k)$ played on $P_n$ is
\[GF_{\Encis(k),P_n}(x,y,t)=\frac{1-t+xt+yt-xt^{k+1}-yt^{k+1}}{(1-t)(1-t-xt^{k+1}-yt^{k+1})}.\]
The total number of positions is generated by
\[GF_{\Encis(k),P_n}(1,1,t)=\frac{1+t-2t^{k+1}}{(1-t)(1-t-2t^{k+1})}.\]
\end{proposition}
\begin{proof}
In \Encis($k$), no pair of vertices at distance $1,2,\ldots,k$ can be coloured simultaneously. Therefore, the positions on a path are exactly those that satisfy the following pattern:
\begin{enumerate}
	\item Starts with zero or more empty vertices;
	\item followed by repeated patterns taken from
	\begin{itemize}
		\item a blue vertex followed by $k$ empty
		\item a red vertex followed by $k$ empty
	\end{itemize}
	followed by zero or more empty vertices; and
	\item ends with
	\begin{itemize}
		\item a blue vertex followed by $0$ to $k-1$ empty,
		\item a red vertex followed by $0$ to $k-1$ empty, or
		\item nothing added.
	\end{itemize}
\end{enumerate}
For example, for $\Encis(2)$ this gives the regular expression\[ E^*[(B|R)EEE^*]^*(B|BE|R|RE|\epsilon ) .\]

For the general case, we use the notation $S_B^{k-1}$ to represent the string \[B|BE|BEE|\ldots|BE^{k-1}.\] Similarly, we set \[S_R^{k-1}=R|RE|REE|\ldots|RE^{k-1}.\] The regular expression for $\Encis(k)$ is then 
\[E^*[(B|R)E^kE^*]^*[S_B^{k-1}|S_R^{k-1}|\epsilon].\] The corresponding expression in the generating function to the term $S_B^{k-1}$ is \[xt \frac{1-e^kt^k}{1-et},\] similarly for $S_R^{k-1}$ we get  \[yt \frac{1-e^kt^k}{1-et}.\]
The generating function taking empty vertices into account is then
\begin{align*}
	GF_{\Encis(k)}(e,x,y,t)&=\left(\frac{1}{1-et}\right)\left(\frac{1}{1-(\frac{xe^kt^{k+1}}{1-et}+\frac{ye^kt^{k+1}}{1-et})}\right)\\
	&\qquad\quad \times\left(\frac{xt(1-e^kt^k)}{1-et}+\frac{yt(1-e^kt^k)}{1-et}+1\right),
\end{align*}
and setting $e=1$ gives \[GF_{\Encis(k)}(x,y,t)=\frac{1-t+xt+yt-xt^{k+1}-yt^{k+1}}{(1-t)(1-t-xt^{k+1}-yt^{k+1})}.\]
\[GF_{\Encis(k)}(1,1,t)=\frac{1+t-2t^{k=1}}{(1-t)(1-t-2t^{k+1}}\qedhere\]
\end{proof}

\subsection{$\Ensnort(k)$}

To find the generating function for the polynomial profile of $\Ensnort(k)$, we will introduce the following shorthand notation for the regular expression:
\begin{align*}
\mathbf{S}_B^k&=B|BE|BEE|...|BE^k\\
\mathbf{S}_R^k&=R|RE|REE|...|RE^k\\
\mathbf{T}_B^k&=B(\mathbf{S}_B^k)^*|B(\mathbf{S}_B^k)^*E|\ldots|B(\mathbf{S}_B^k)^*E^{k-1}\\
\mathbf{T}_R^k&=R(\mathbf{S}_R^k)^*|R(\mathbf{S}_R^k)^*E|\ldots|R(\mathbf{S}_R^k)^*E^{k-1}.
\end{align*}

The regular expression for $\Ensnort(k)$ is then
\[E\mbox{*}[(B(\mathbf{S}_B^k)\mbox{*} | R(\mathbf{S}_R^k)\mbox{*}))E^kE\mbox{*}]\mbox{*}[\mathbf{T}_B^k |\mathbf{T}_R^k|\epsilon].\]

\begin{example}
Consider $\Ensnort(2)$ played on $P_n$. The regular expression for $\Ensnort(2)$ is 
\begin{align*}
E^*&[B(BE|B)^*|R(RE|R)^*|EE^*]^*.\\
&\qquad[B(BE|B)^*|B(BE|B)^*E|R(RE|R)^*|R(RE|R)^*E|\epsilon ],
\end{align*}
giving that the generating function for the polynomial profile is
\begin{align*}
	GF_{\Ensnort(2)}&(x,y,t)=\left(\frac{1}{1-t}\right)\left(\frac{1}{1-(\frac{xt}{1-xt-xt^2}\frac{yt}{1-yt-yt^2})(\frac{t^2}{1-t})}\right)\\
	&\left(\frac{xt}{1-xt-xt^2}+\frac{xt^2}{1-xt-xt^2}+\frac{yt}{1-yt-yt^2}+\frac{yt^2}{1-yt-yt^2}+\epsilon\right)
\end{align*}
The first few polynomial profiles are in \cref{Tab:initialEnsnort}.
\begin{table}[!ht]
\begin{center}
\begin{tabular}{| c | c | c | c |}
\hline
n & $P_{\Ensnort(2),P_n}(x,y)$ & $P_{\Ensnort(2),P_n}(x)$ & $P_{\Ensnort(2),P_n}(1)$ \\
\hline
$0$ & $1$ & $1$ & $1$ \\
$1$ & $1+1x+1y$ & $1+2x$ & $3$ \\
$2$ & $1+2x+2y+x^2+y^2$ & $1+4x+2x^2$ & $7$ \\
$3$ & $1+3x+3y+3x^2+3y^2+x^3+y^3$ & $1+6x+6x^2+2x^3$ & $15$ \\
\hline
\end{tabular}
\end{center}
\caption{First few initial terms for the recursion of the polynomial profile of $\Ensnort(2)$}\label{Tab:initialEnsnort}
\end{table}
\end{example}

In general, we get the following.
\begin{proposition}
The generating function for the polynomial profile of $\Ensnort(k)$ on $P_n$ is
\begin{align*}
	GF_{\textsc{EnSnort(k)},P_n}&(x,y,t)=\left(\frac{1}{1-t}\right)\left(\frac{1}{\displaystyle 1-\frac{xt}{\displaystyle 1-\sum_{i=1}^kxt^i}+\frac{yt}{\displaystyle 1-\sum_{i=1}^kyt^i}(\frac{t^k}{1-t})}\right) \\
	&\times\left(\sum_{n=1}^k\frac{xt^n}{\displaystyle 1-\sum_{i=1}^kxt^i}+\sum_{n=1}^k\frac{yt^n}{\displaystyle 1-\sum_{i=1}^kyt^i}+1\right).
\end{align*}
\end{proposition}

\subsection{$\Cis_2$}

In this section we find the recursion for the polynomial profile and number of positions of $\Cis_2$ played on paths, as well as the generating function for the polynomial profile.
\begin{proposition}
The polynomial profile for $\Cis_2$ played on $P_n$ is given by 
\begin{align*}
	 P_{\Cis_2,P_{n}}(x,y)&=P_{\Cis_2,P_{n-1}}(x,y)+(x+y)P_{\Cis_2,P_{n-3}}(x,y) \\
	&\qquad+(x^2+y^2+2xy)P_{\Cis_2,P_{n-4}}(x,y),
\end{align*}
for $n\geq 4$.

The total number of positions\footnote{OEIS sequence \href{https://oeis.org/A138495}{A138495}} is recursively given by \[P_{\Cis_2,P_{n}}(1)=P_{\Cis_2,P_{n-1}}(1)+2P_{\Cis_2,P_{n-3}}(1)+4P_{\Cis_2,P_{n-4}}(1).\]

The initial terms are given in \cref{Tab:initialCis2}.

\begin{table}[!ht]
\begin{center}
\begin{tabular}{| c | c | c | c |}
\hline
n & $P_{\Cis_2,P_n}(x,y)$ & $P_{\Cis_2,P_n}(x)$ & $P_{\Cis_2,P_n}(1)$ \\
\hline
$0$ & $1$ & $1$ & $1$ \\
$1$ & $1+x+y$ & $1+2x$ & $3$ \\
$2$ & $1+2x+2y+x^2+y^2+2xy$ & $1+4x+3x^2$ & $9$ \\
$3$ & $1+3x+3y+2x^2+2y^2+4xy$ & $1+6x+8x^2$ & $15$ \\
\hline
\end{tabular}
\end{center}
\caption{Initial terms for the recursion of the polynomial profile of $\Cis_2$}\label{Tab:initialCis2}
\end{table}
\end{proposition}

\begin{proof}
The initial cases can be checked computationally.

For the recursion, consider the leftmost vertex of the path. If this vertex is uncoloured, the rest of the $n-1$ vertices can be any legal position for this game on $P_{n-1}$, which gives $P_{\Cis_2,P_{n-1}}(x,y)$ as a term in the recursion. On the other hand, if this vertex is coloured blue or red, then either the next two vertices in the path are both empty or the next vertex is also coloured blue or red, followed by the next two empty. The first case leaves any legal position on $P_{n-3}$, which gives $xP_{\Cis_2,P_{n-3}}(x,y)$ for blue and $yP_{\Cis_2,P_{n-3}}(x,y)$ for red. The second case results in two coloured vertices followed by two empty vertices, leaving any legal position on $P_{n-4}$. In this case, the two leftmost vertices can be any combination of one blue and one red, two blue, or two red, giving us the term $(x^2+y^2+2xy)P_{\Cis_2,P_{n-4}}(x,y)$. 

This gives the recursion
\begin{align*}
	 P_{\Cis_2,P_{n}}(x,y)&=P_{\Cis_2,P_{n-1}}(x,y)+(x+y)P_{\Cis_2,P_{n-3}}(x,y) \\
	&\qquad+(x^2+y^2+2xy)P_{\Cis_2,P_{n-4}}(x,y).
\end{align*}

Setting $x=y=1$, we get the recursion for the total number of positions as desired.
\end{proof}

Using the shorthand $\mathbf{U}=BB|BR|RB|RR|B|R$, the regular expression is
\[E^*(\mathbf{U}EEE^*)^*(\mathbf{U}|\mathbf{U}E|\epsilon).\]

Thus we get the following:
\begin{proposition}
The generating function for the polynomial profile of $\Cis_2$ played on paths is
\begin{align*}
GF_{\Cis_2,P_n}(x,y,t)
&=\frac{1+xt+yt+(xt+yt)^2+et(xt+yt+(xt+yt)^2)}{1-et-e^2t^2(xt+yt+(xt+yt)^2)}.
\end{align*}
The total number of positions is generated by
\[GF_{\Cis_2,P_n}(1,1,t)=\frac{1+2t+6t^2+4t^3}{1-t-2t^3-4t^4}.\]
\end{proposition}

\section{\Col, \Snort, and \cis}\label{csc}

\subsection{\cis on Cycles}

We will generalize the results of \cite{profiles} for \cis on paths to cycles. We will find the polynomial profiles using the polynomial profiles on paths, as well as the generating function for the polynomial profile.

\begin{proposition}
The polynomial profile when playing \textsc{Cis} on $C_n$ for $n\geq 4$ is given by 
\[P_{Cis,C_n}(x,y)=P_{\text{Cis},P_{n-1}}(x,y)+(x+y) P_{\textsc{Cis},P_{n-2}}(x,y).\]   

The number of positions is $ 2^{n} +(-1)^n$.
\end{proposition}

\begin{proof}
We fix one vertex in the cycle and label it $v_0$. We then label the rest of the vertices as $v_1, \ldots, v_{n-1}$ in a clockwise rotation. 

If $v_0$ is coloured, then the adjacent vertices ($v_1$ and $v_{n-1}$) must be empty. This leaves any possible position on $P_{n-3}$ which gives $xP_{\textsc{Cis},P_{n-3}}(x,y)$ if it is a blue vertex and $yP_{\textsc{Cis},P_{n-3}}(x,y)$ if it is a red vertex.

If $v_0$ is uncoloured, this leaves any possible position on $P_{n-1}$, thus contributing the term $P_{\text{Cis},P_{n-1}}(x,y)$.

This gives the recursion \[P_{Cis,C_n}(x,y)=(x+y) P_{\textsc{Cis},P_{n-3}}(x,y) + P_{\text{Cis},P_{n-1}}(x,y).\]
We then set $x=y=1$ to get the number of positions recursively as
\[P_{Cis,C_n}(1)=2 P_{\textsc{Cis},P_{n-3}}(1) + P_{\text{Cis},P_{n-1}}(1).\]

In \cite{profiles}, the number of positions for \textsc{Cis} on paths was found to be $P_{\Cis,P_n}(1)=\frac{2^n-(-1)^n}{3}$. Substituting this into the recursion, we get
\begin{align*}
	P_{\Cis,C_n}(1)& = 2P_{\Cis,P_{n-3}}(1)+P_{\Cis,P_{n-1}}(1)\\
	&=2\left( \frac{2^{n-3}-(-1)^{n-3}}{3}\right) + \left( \frac{2^{n-1}-(-1)^{n-1}}{3}\right) \\
	&=2^{n-2}+(-1)^n.\qedhere
\end{align*}
\end{proof}

\begin{proposition}
The generating function for the polynomial profile of \cis played on cycles is
\[GF_{\Cis,C_n}(1,x,y,t)=\frac{1+x^2t^2+y^2t^2}{1-t-xt^2-yt^2},\]
and thus the generating function for the total number of positions is 
\[GF_{\Cis,C_n}(1,1,1,t)=\frac{1+2t^2}{1-t-2t^2}.\]
\end{proposition}
\begin{proof}
For \cis on a cycle, we will fix a vertex $v$ and will consider a position on $C_n$ as a position on $P_{n+1}$, where the first and last vertex are simultaneously either empty, blue, or red. Recall that no adjacent vertices can both be coloured. Therefore, the positions on the equivalent path $P_{n+1}$ are exactly those that satisfy one of the following pattern:
\begin{enumerate}
	\item Starts with an empty vertex;
	\begin{itemize}
		\item followed by zero or more empty vertices;
		\item followed by repeated patterns taken from
		\begin{itemize}
			\item a blue vertex followed by an empty vertex
			\item a red vertex followed by an empty vertex
		\end{itemize}
		followed by zero or more empty vertices; or
	\end{itemize}
	\item Starts with a blue vertex;
	\begin{itemize}
		\item followed by at least one empty vertex;
		\item followed by repeated patterns taken from
		\begin{itemize}
			\item a blue vertex followed by an empty vertex
			\item a red vertex followed by an empty vertex
		\end{itemize}
		followed by zero or more empty vertices;
		\item ends with a blue vertex; or
	\end{itemize}
	\item Starts with a red vertex;
	\begin{itemize}
		\item followed by at least one empty vertex;
		\item followed by repeated patterns taken from
		\begin{itemize}
			\item a blue vertex followed by an empty vertex
			\item a red vertex followed by an empty vertex
		\end{itemize}
		followed by zero or more empty vertices;
		\item ends with a red vertex.
	\end{itemize}
\end{enumerate}

The regular expression for \cis then is
\[[EE^*(BEE^*|REE^*)^*]|[BEE^*(BEE^*|REE^*)^*(B)]|[REE^*(BEE^*|REE^*)^*(R)].\]

The generating function for the three cases count the first and last vertex separately, while they are the same for us on the cycle. To adjust for this, we divide the generating function for the first case by $et$, the second by $xt$, and the third by $yt$. This gives us the following generating function for the cycle $C_n$: 
\begin{align*}
	GF_{\Cis,C_n}(e,x,y,t)&= \left(\frac{1}{1-et}\right)\left(\frac{1-et}{1-et-xet^2-yet^2}\right)\\
	&\qquad + \left(\frac{xet^2}{1-et}\right)\left(\frac{1-et}{1-et-xet^2-yet^2}\right) \\
	&\qquad + \left(\frac{yet^2}{1-et}\right)\left(\frac{1-et}{1-et-xet^2-yet^2}\right).
\end{align*}
Setting $e=1$ and simplifying gives the desired generating function.
\end{proof}

\subsection{Recursion for \col and \snort on Cycles}

We will give the recursions for the total number of positions of \col and \snort on cycles. The recursion for the polynomial profile can be found similarly. As the latter do not simplify nicely, we will only give the former.

\begin{proposition}\label{col(1)}
The number of positions\footnote{This sequence seems to be the OEIS sequence \href{https://oeis.org/A051927}{A051927}. We have confirmed this for $n \leq 12$.} when playing \col on $C_n$ is given by \[P_{\Col,C_n}(1)=P_{\Col,P_{n-1}}(1)+3P_{\Col,P_{n-3}}(1)+2P_{\Col,P_{n-4}}(1)+P_{\Col,C_{n-2}}(1).\]
\end{proposition}
\begin{proof}We fix one vertex in the cycle and label it $v_0$. We then label the rest of the vertices as $v_1, \ldots, v_{n-1}$ in a clockwise rotation. 

If $v_0$ is uncoloured, this leaves any possible position on $P_{n-1}$, this gives us $P_{\text{Col},P_{n-1}}(1)$ possible positions.

If $v_0$ is coloured, then there are three cases to consider:
\begin{enumerate}
	\item both $v_1$ and $v_{n-1}$ are empty;
	\item both $v_1$ and $v_{n-1}$ are coloured, necessarily the same colour as $v_0$; and
	\item either $v_1$ or $v_{n-1}$ is coloured, necessarily the same colour as $v_0$, with the other being empty. 
\end{enumerate}

The first case leaves any possible position on $P_{n-3}$, whether $v_0$ is blue or red, thus there are $2P_{\text{Col},P_{n-3}}(1)$ possible positions for this case. 

The second case can be reduced to a cycle on $n-2$ vertices where the fixed vertex cannot be empty. This gives us $P_{\text{Col},C_{n-2}}(1)-P_{\text{Col},P_{n-3}}(1)$ possible positions. 

The third case reduces to a coloured path on $n-2$ vertices. Let $f_B(n)$ denote the polynomial enumerating the positions when playing \col on $C_n$ with one vertex coloured blue, and similarly for $f_R(n)$ and $f_E(n)$. This gives us 
\begin{align*}
f_E(n-3)(1)&+f_B(n-3)(1)+f_E(n-3)(1)+f_R(n-3)(1)\\
&=P_{\text{Col},C_{n-3}}(1)-f_R(n-3)(1)+P_{\text{Col},C_{n-3}}(1)-f_B(n-3)(1)\\
&=2P_{\text{Col},C_{n-3}}(1)-(f_R(n-3)(1)+f_B(n-3)(1))\\
&=2P_{\text{Col},C_{n-3}}(1)-(P_{\text{Col},C_{n-3}}(1)-P_{\text{Col},C_{n-4}}(1))\\
&=P_{\text{Col},C_{n-3}}(1)+P_{\text{Col},C_{n-4}}(1). 
\end{align*}
We count this case twice, once for $v_0$ being empty and $v_{n-1}$ being coloured and once for the other way around, this gives us $2P_{\text{Col},C_{n-3}}(1)+2P_{\text{Col},C_{n-4}}(1)$.

All together we get \[P_{\text{Col},C_{n}}(1)=P_{\text{Col},P_{n-1}}(1)+3P_{\text{Col},P_{n-3}}(1)+2P_{\text{Col},C_{n-4}}(1)+P_{\text{Col},C_{n-2}}(1).\qedhere\]
\end{proof}

The proof for \snort on cycles is similar to the proof for \col on cycles and is thus omitted. 

\begin{proposition}The number of positions\footnote{This sequence seems to be OEIS sequence \href{https://oeis.org/A124696}{A124696}. We have confirmed this for $n \leq 13$.} when playing \snort on $C_n$ is given by \[P_{\Snort,C_n}(1)=P_{\Snort,P_{n-1}}(1)+3P_{\Snort,P_{n-3}}(1)+2P_{\Snort,P_{n-4}}(1)+P_{\Snort,C_{n-2}}(1).\] 
\end{proposition}

\subsection{\Col, \Snort, and \cis on Complete Bipartite Graphs}

When playing \Col, \Snort, or \cis on the star $K_{1,n}$, the state of the central vertex will completely determine the possible positions. 

For \col and \Snort, if the central vertex is coloured blue, the outer vertices can each independently be red or empty, respectively blue or empty. Similarly if it is coloured red. If the central vertex is uncoloured, the outer vertices can independently be empty or either colour.

For \Cis, if the central vertex is coloured, the outer vertices have to be empty, while if it is uncoloured, then the outer vertices can independently be empty or either colour.

This gives us the following result.

\begin{proposition}
When playing on a star, the total number of positions is given by
\begin{align*}
P_{\Col, K_{1,n}}(1)&=2^{n+1}+3^n\\
P_{\Snort, K_{1,n}}(1)&=2^{n+1}+3^n\\
P_{\Cis, K_{1,n}}(1)&=2+3^n.
\end{align*}
\end{proposition}

That the number of positions for \col and \snort on $K_{1,n}$ is the same is not a surprise. It is known that this is the case for any bipartite board:

\begin{theorem}[Theorem 4.2 in \cite{profiles}]\label{doppelganger}
If $B$ is a bipartite graph, then $P_{\Col,B}(x)=P_{\Snort,B}(x)$. In particular, $P_{\Col,B}(1)=P_{\Snort,B}(1)$, \ie , the number of positions is the same for \col and \snort when playing on a bipartite graph. 
\end{theorem}

We have computed the number of positions on other complete bipartite graphs for \col and \snort and these can be found in \cref{Tab:KmnSnortCol}.

\begin{table}[!ht]
\resizebox{\textwidth}{!}{\begin{tabular}{c | c | c | c | c | c | c | c | c | c | c | c | c | c | c }
$m/n$ & 0 & 1 & 2 & 3 & 4 & 5 & 6 & 7 & 8 & 9 & 10 & 11 & 12 & 13\\
\hline
0 & 1 & 3 & 9 & 27 & 81 & 243 & 729 & 2187 & 6561 & 19683 & 59049 & 177147 & 531441 & 1594323\\
\hline
1 & 3 & 7 & 17 & 43 & 113 & 307 & 857 & 2443 & 7073 & 20707 & 61097 & 181243 & 539633 & \\
\hline
2 & 9 & 17 & 35 & 77 & 179 & 437 & 1115 & 2957 & 8099 & 22757 & 65195 & 189437 &  & \\
\hline
3 & 27 & 43 & 77 & 151 & 317 & 703 & 1637 & 3991 & 10157 & 26863 & 73397 & & & \\
\hline
4 & 81 & 113 & 179 & 317 & 611 & 1253 & 2699 & 6077 & 14291 & & & & \\
\hline
5 & 243 & 307 &437 & 703 & 1253 & 2407 & 4877 & 10303 & & & & & & \\
\hline
6 & 729 &  857 & 1115 & 1637 & 2699 & 4877 & 9395 & & & & & & & \\
\hline
7 & 2187 & 2443 & 2957 & 3991 & 6077 & 10303 & & & & & & & & \\
\hline
8 & 6561 & 7073 & 8099 & 10157 & 14291 & & & & & & & & & \\
\hline
9 & 19683 & 20707 & 22757 &26863 & & & & & & & & & & \\
\hline
10 & 59049 & 61097 &  65195 & 73397 & & & & & & & & & & \\
\hline
11 & 177147 & 181243 & 189437 & & & & & & & & & & & \\
\hline
12 & 531441 & 539633 & & & & & & & & & & & & \\
\hline
13 & 1594323 & & & & & & & & & & & & & \\
\end{tabular}}
\caption{The number of positions when playing \col or \snort on $K_{m,n}$}\label{Tab:KmnSnortCol}
\end{table}

Based on this data, we pose the following conjecture.

\begin{conjecture}\label{ConjBipartite}
The number of positions when playing \col or \snort on the complete bipartite graph $K_{m,n}$ are recursively given by
\[P_{\Col,K_{m,n}}(1)=5P_{\Col,K_{m,n-1}}(1)-6P_{\Col,K_{m,n-2}}(1)+c_m\]
with initial terms as per \cref{Tab:KmnSnortCol} and $c_m$ is given by the OEIS sequence \href{https://oeis.org/A260217}{A260217} (first few terms are $c_2=4$, $c_3=24$, $c_4=100$, $c_5=360$, and $c_6=1204$).
\end{conjecture}

For \Cis, similarly to the situation of the star, as soon as a single vertex is coloured, only vertices in the same part of $K_{m,n}$ are able to be coloured. Ensuring that we do not count the empty position twice, we get the following result. 
\begin{proposition}
The number of positions when playing \cis on the complete bipartite graph $K_{m,n}$ is given by
\[P_{\Cis,K_{m,n}}(1)=3^m+3^n-1.\]
\end{proposition}

\section{Future Work}\label{ques}
As was the case in \cite{profiles} for \Col, $\Encol(k)$ appears the most complicated of the generalizations of \Col, \Snort, and \cis that we considered. We are still looking into finding a recursion, as well as a generating function, for $\Encol(k)$ played on paths. 

For \col and \snort played on cycles we have found recursion for the polynomial profiles and the number of positions, but have not yet found the generating functions. For complete bipartite graphs, in addition to trying to prove \cref{ConjBipartite}, we also would like to find generating functions for this case.

Of course, this work could also be extended by considering other distance games or looking at the generalizations when played on boards other than paths.

Finally, we are interested in the ratio of positions in purely alternating play, which can also be found from the polynomial profile, to the total number of positions.

\bibliographystyle{plain}
\bibliography{paper_bibl}
\end{document}